\documentclass[11pt]{amsart}
 
\usepackage{amssymb,graphicx, amsmath, amsthm,MnSymbol, array}

\graphicspath{ {./images/} }
\usepackage{tikz-cd}
\DeclareMathAlphabet\mathbfcal{OMS}{cmsy}{b}{n}

\usepackage{todonotes}

\usepackage{floatflt}

\usepackage{hyperref}

\usepackage{ulem}
\usepackage{mathbbol}

\usepackage[inline]{enumitem}

\usepackage[mathscr]{eucal}

\usepackage[font=small]{caption}

\usepackage{etoolbox}

\makeatletter
\def\l@subsection{\@tocline{2}{0pt}{2pc}{5pc}{}}
\makeatother

\def\m{{\mathfrak m}} 

\def\O{{\mathcal{O}}}

\def\PP{f}

\def\PP{\overline{\PP}}

\newtheorem{theorem}{Theorem}[section]
\newtheorem{lemma}[theorem]{Lemma}
\newtheorem{proposition}[theorem]{Proposition}
\newtheorem{corollary}[theorem]{Corollary}

\theoremstyle{definition}

\newtheorem{definition}[theorem]{Definition}

\newtheorem{remark}[theorem]{Remark}
\newtheorem{example}[theorem]{Example}

\long\def\alert#1{\smallskip{\hskip\parindent\vrule%
\vbox{\advance\hsize-2\parindent\hrule\smallskip\parindent.4\parindent%
\narrower\noindent#1\smallskip\hrule}\vrule\hfill}\smallskip}

\def\ann{\mathrm{ann}}

\def\ff{\mathfrak}

\def\Spec{\mbox{\rm Spec}}

\def\ker{{\mbox{\rm{Ker} }}}
\def\Max{\mbox{\rm Max}}

\def\Id{\mbox{\rm Id}}

\def\PP{{\mathcal P}}

\def\cal{\mathcal}

\begin{document}

\title{Realization of spaces of commutative rings}

\author{Laura Cossu}
\address{\parbox{\linewidth}{(Laura Cossu) University of Cagliari\\
Department of Mathematics and Computer Science\\
Via Ospedale 72, 
09124 Cagliari, Italy}}
\email{laura.cossu3@unica.it}

\author{Bruce Olberding}
\address{\parbox{\linewidth}{(Bruce Olberding) New Mexico State University\\ Department of Mathematical Sciences\\ 1290 Frenger Mall, Las Cruces, New Mexico 88003-8001, USA}}
\email{bruce@nmsu.edu}

\begin{abstract} Motivated by recent work on the  use of topological methods to study collections of rings between an integral domain and its quotient field, we examine spaces of subrings of a commutative ring, endowed  with the Zariski or patch topologies. We introduce three notions to study such a space $X$: patch bundles, patch presheaves and patch algebras. When $X$ is compact and Hausdorff, patch bundles give a way to approximate $X$ with topologically more tractable spaces, namely Stone spaces. Patch presheaves encode the space $X$ into stalks of a presheaf of rings over a Boolean algebra, thus giving a more geometrical setting for studying $X$. To both objects, a patch bundle and a patch presheaf, we associate what we call a patch algebra, a commutative ring that efficiently realizes the rings in $X$ as factor rings, or even localizations, and whose structure reflects various properties of the rings in $X$. 
\end{abstract}


\thanks{The first author acknowledges the financial support from the European Union's Horizon 2020 program (Marie Sklodowska-Curie grant 101021791) and the Austrian Science Fund FWF (grant DOI 10.55776/ P\-AT\-975\-6623). The second author was supported by NSF grant DMS-2231414. L.~C.~is a member of the National Group for Algebraic and Geometric Structures and their Applications (GNSAGA), a department of the Italian Mathematics Research Institute (INdAM)}

\maketitle

\tableofcontents

\section{Introduction}

All algebras considered in the paper are commutative and unital, and all algebra homomorphisms are unital. This article is motivated by a number of recent studies of spaces of rings between an integral domain $D$ and its quotient field $F$.   For example, the space in question could be that of all the valuation rings lying between $D$ and $F$, known as the Zariski-Riemann space of $F/D$ \cite{FFL,FFL3, HLOT, OZar, OIrr, Spi1, Spi2, Spi3}.  Alternatively, one might consider the space of all integrally closed rings or the space of all local rings between $D$ and $F$  \cite{Fin, FFS, FS, ONoeth}. One could also examine the set of local rings in a blow-up along a specific ideal of $D$, such as in \cite{OZar}. In all these examples, the space is a spectral space with respect to the Zariski topology. Spectral spaces are precisely the spaces that arise as the prime spectra of commutative rings, and so there is often in the study of these spaces an emphasis on realizing the space via a ring. Typically, this realization involves the prime ideals of a commutative ring.  In this article, we seek a different sort of realization, one which realizes the rings in the spectral space as factor rings (and often localizations) of a single commutative ring. For example, we exhibit a commutative ring whose localizations at the maximal ideals are precisely the valuation rings in the Zariski-Riemann space of $F/D$; see Example~\ref{example: ZS}. In working with spaces of overrings instead of  prime ideals, there is no restriction on the types of spectral spaces that can occur: For each  spectral space $X$, there is a domain $D$ such that $X$  is homeomorphic to a space of rings between $D$ and  its quotient field (see Proposition~\ref{not special}).

Exhibiting a single ring that has all the rings of a space $X$ of overrings of $D$ as factor rings is not difficult to achieve: simply take the Cartesian product of the rings. However, this approach is highly inefficient, and if the space $X$ is infinite, the product will have an abundance of extra rings as factor rings. We propose a more efficient construction that we call patch algebra. Our new approach patches together all the rings of the space $X$ into a single algebra with the property that (at least if the rings in $X$ are between a domain $D$ and its quotient field $F$) its factors modulo its minimal prime ideals are, up to isomorphism, precisely the rings in $X$. By working with the Pierce spectrum of the patch algebra rather than the minimal spectrum, we can considerably relax the restrictions on $D$ and $F$, allowing zero divisors and moving outside of the birational case. Regardless of whether the rings in the spaces we consider have zero-divisors, the patch algebra itself will have many zero-divisors, and in fact many idempotents. It is the idempotents that ultimately patch the rings in $X$ into a single ring. 

Our methods apply to spaces of subrings of a given a ring $R$, a situation that subsumes the examples mentioned above. The main topology we consider on a space of subrings of $R$ is the Zariski topology. However, we also work with the patch (or constructible) topology,  a refinement of the Zariski topology that has a basis of clopen sets. We study spaces of subrings of $R$ that are closed in the patch topology and their associated patch algebras. While these patch-closed spaces are the most directly tractable for our purposes, we introduce the notion of patch bundle to handle more general classes of spaces, including compact Hausdorff spaces of subrings and spectral spaces that are not necessarily patch-closed in the larger space of all subrings of $R$. The idea behind the patch bundle is to approximate a non-patch-closed space $X$ with a continuous map from a Stone space onto~$X$. This is reminiscent of the approach taken in condensed mathematics, where a subspace of a topological space is replaced with the set of continuous maps from a Stone space into this subspace; see \cite[Example 1.5, p.~8]{Sch}. 

Not surprisingly, our construction of a patch algebra has sheaf-theoretic aspects. We wait until the sequel of this article to fully develop this point of view, but in the present paper we do take a step in this direction by associating to each patch bundle of subrings of $R$ what we call a patch presheaf, which is a presheaf of rings on a Stone space. We show in Section~\ref{sec: 4 A category equivalence} that this association is functorial and is part of a category equivalence between patch bundles and patch presheaves. Thus, we ultimately have three ways to manifest spaces of subrings of a ring $R$: topologically as patch bundles, geometrically as patch presheaves, and algebraically as patch algebras. 

After reviewing some preliminary notions in Section~\ref{sec: 2 notation and preliminary results}, we introduce patch spaces, patch presheaves and patch bundles of subrings   in Section~\ref{sec: 3 patch spaces, presheaves, bundles}. We discuss some examples of patch spaces from the literature, and we show how the notion of patch bundle generalizes that of patch space and is versatile enough to capture compact Hausdorff spaces of rings and spectral spaces of rings. 
In Section~\ref{sec: 4 A category equivalence}, we prove that the category of patch presheaves is equivalent to that of patch bundles, thus showing that these two notions are ultimately different expressions of the same concept. Since patch bundles are more general than patch spaces, the question arises as to which patch presheaves correspond to patch spaces when viewed as subobjects in the category of patch bundles. This question is answered in Section~\ref{sec: 5 patch spaces and distinguished patch presheaves}, using what we call distinguished patch presheaves.
In Section~\ref{sec: 6 Patch algebras} we introduce patch algebras, which are algebras constructed from patch presheaves (or patch bundles), and develop their properties. We show how for a ring $R$, the space of subrings that is the image of a patch bundle is encoded into the patch algebra. In Section~\ref{sec: 7 structure of patch algebras}, we further investigate the structure of patch algebras and show that with mild restrictions on $R$ (such as that $R$ is indecomposable or a domain), the rings that are in the image of the bundle become particularly easy to extract from the patch algebra, as quotients by minimal primes or as localizations at maximal ideals. 

We leave further applications of these ideas for a sequel to this paper, in which we work out descriptions of the patch algebra for specific classes of rings.  

\section{Notation and preliminary results}\label{sec: 2 notation and preliminary results}

Throughout, $R$ will denote a commutative ring with $1$. We start the section by formalizing the notation and terminology that will be used in the paper.
Let $A$ be an $R$-algebra and let $\Id(A)$ be the set of idempotents of $A$. Since $A$ is a commutative ring with $1$, it is well known that $\Id(A)$ is a Boolean algebra with join $e \vee f = e + f -ef$, meet $e \wedge f = ef$, top element $1$ and bottom element $0$. For an element $a\in A$, we let $\ann_{R}(a)=\{r\in R : ra=0\}$ be the {\it annihilator ideal} of $a$ in $R$. The $R$-algebra $A$ is {\it idempotent generated} if it is generated as an $R$-algebra by a Boolean subalgebra $B$ of $\Id(A)$. In this case, $B$ generates $A$ both as an $R$-algebra and an $R$-module and every element of $A$ is an $R$-linear combinations of idempotents of $B$. A set $E=\{e_1, \dots, e_n\}$ of nonzero idempotents of $A$ is said to be {\it orthogonal} if $e_i\wedge e_j=0$ for every $i\ne j$. If, in addition, $\bigvee_{i=1}^n e_i=1$, then $E$ is a {\it full orthogonal} set. Accordingly, we say that $a\in A$ has an {\it orthogonal decomposition} (resp., a {\it full orthogonal decomposition}) if $a=\sum_{i=1}^n t_ie_i$, with $t_i\in R$ and $\{e_1, \dots, e_n\}$ an orthogonal set (resp., a full orthogonal set) of idempotents. By possibly adding a term with a $0$ coefficient, it is always possible to turn an orthogonal decomposition into a full orthogonal one. Recall, finally, that a ring $R$ is {\it indecomposable} if $\Id(R)=\{0,1\}$.

\subsection{Ideals of Boolean algebras}\label{subsect: ideals of Boolean algebras}  Following \cite[Chapter~18]{GiHa09}, recall that an ideal of a Boolean algebra $B$ is a subset $I\subseteq B$ such that $0\in I$, $x\vee y\in I$ for every $x,y\in I$, and $a\wedge x\in I$ for every $a\in B$ and $x\in I$. An ideal $I$ of $B$ is {\it proper} if it is properly included in $B$. For any subset $E$ of a Boolean algebra $B$, the intersection of all the ideals of $B$ containing $E$ is an ideal of $B$ called the {\it ideal generated by $E$}. We will need later the following result, that we recall without proof.

\begin{proposition}[{\cite[Theorem 11, p.~155]{GiHa09}}]\label{prop: generated ideals}
An element $e$ of a Boolean algebra is in the ideal generated by a set $E$ if and only if there is a finite subset $F$ of $E$ such that $e \le \bigvee_{f\in F} f$.
\end{proposition}

A proper ideal of a Boolean algebra $B$ is {\it maximal} if it is not properly included in any other proper ideal of $B$.
We collect in the next proposition some facts about maximal ideals of Boolean algebras that will be useful in the following discussion.

\begin{proposition}\label{prop: max ideals} Let $B$ be a Boolean algebra.
\begin{enumerate}[label=\textup{(\arabic{*})}, mode=unboxed]
    \item\label{prop: max ideals 1} A proper ideal $\m$ of $B$ is maximal if and only if, for every $e\in B$, either $e$ or its complement is in $\m$ {\cite[Lemma 1, p.~171]{GiHa09}}.

    \item\label{prop: max ideals 2} A proper ideal $\m$ of $B$ is maximal if and only if it is prime, i.e., if and only if $e\wedge f\in \m$ implies $e\in\m$ or $f\in \m$ {\cite[Corollary 1, p.~172]{GiHa09}}.

    \item\label{prop: max ideals 3} Every proper ideal in $B$ is included in a maximal ideal {\cite[Theorem 12, p.~172]{GiHa09}}.

    \item\label{prop: max ideals 4} Every ideal in $B$ is the intersection of the maximal ideals in which it is contained {\cite[Exercise 3, p.~174]{GiHa09}}.
\end{enumerate}
    
\end{proposition}

\subsection{Stone spaces of Boolean algebras}\label{subsec: Stone spaces} A {\it Stone space} $X$ is a compact Hausdorff topological space that is also zero-dimensional, meaning that its clopen subsets form a basis of its topology. Equivalently, $X$ is a Stone space if it is compact, Hausdorff, and totally disconnected. For other characterizations of Stone spaces we refer the reader to \cite[Chapter II.4]{Joh86}. It is well known \cite{Sto36} that every Boolean algebra $B$ has an associated Stone space: the set  ${\rm Max}(B)$ of  maximal ideals of $B$, whose topology is generated by the sets
$$\{\m\in {\rm Max}(B) : e\not\in \m\},$$
where $e$ is an element of the Boolean algebra $B$. The above sets are clopen in ${\rm Max}(B)$ and they form a Boolean algebra with respect to set containment. We now state a fundamental result for Stone spaces of Boolean algebras, originally proved by Stone in \cite{Sto36} and commonly known as {\it Stone's representation theorem}. We will extensively use this result throughout the paper, and for the proof, we refer to \cite{Sto36} or \cite[Chapter II.4]{Joh86}. 

\begin{theorem}[Stone's representation theorem]\label{thm: Stone's repr. thm}
 Every Boolean algebra $B$ is isomorphic to the algebra of clopen subsets of its Stone space ${\rm Max}(B)$. The isomorphism associates to every $e\in B$ the clopen subset $\{\m\in {\rm Max}(B) : e\not\in \m\}\subseteq {\rm Max}(B).$ Furthermore, every Stone space $X$ is homeomorphic to the Stone space ${\rm Max}({\rm Clop}(X))$ associated to the Boolean algebra of clopen subsets of $X$. The homemorphism associates to each element $x\in X$ the maximal ideal $\m=\{U\in {\rm Clop}(X) : x\not \in U\}$ of the Boolean algebra ${\rm Clop}(X)$.
\end{theorem}

\subsection{Specker algebras}\label{subsec: Specker algebra} There is a natural way to associate to a ring $R$ and a Boolean algebra $B$ an $R$-algebra whose elements are $R$-linear combinations of idempotents corresponding to the elements of $B$. The lemma below is grounded in Section 2 of \cite{BMMO}. 

\begin{lemma}\label{lem: 1.1}
Let $R$ be a ring, and $B$ a Boolean algebra. 
There is an $R$-algebra $R[B]$ and an injective homomorphism of Boolean algebras $\alpha:B \rightarrow \Id(R[B])$ such that:
\begin{enumerate}[label=\textup{(\roman{*})}, mode=unboxed]
    \item\label{lem: 1.1 i} for each $ e \in B$ and $t\in R$,  $t\alpha(e) = 0$ if and only if $t=0$ or $e=0$, and
    \item\label{lem: 1.1 ii} every element of $R[B]$ is of the form $ t_1\alpha(e_1) + \cdots + t_n\alpha(e_n)$, where each $e_i \in B$ and each $t_i \in R$.
\end{enumerate}
 \end{lemma}
\begin{proof}
    If we define $R[B]$ as in \cite[Def.~2.4]{BMMO}, \ref{lem: 1.1 i} and \ref{lem: 1.1 ii} follow from \cite[Lemma~2.6]{BMMO}.
\end{proof}

We can identify $B$ with its image in $R[B]$ and write the elements of $R[B]$ as $t_1e_1 + \cdots + t_ne_n$, where $e_i\in B$ and $t_i\in R$ for every $i$. Standard Boolean algebra arguments show that every element $a$ in $R[B]$ has an orthogonal decomposition, that is, $a =t_1e_1 +\cdots +t_ne_n$, where the $e_i$ are pairwise orthogonal idempotents in $B$ and the $t_i$ are in $R$. Following the terminology in \cite{BMMO}, statement \ref{lem: 1.1 i} of Lemma \ref{lem: 1.1} asserts that each nonzero idempotent $e \in B$ is {\it faithful}, meaning that $\ann_R(e)=\{0\}$. We then recover from \cite[Lemma 2.1]{BMMO} that every element $a\in R[B]$ can be written {\it uniquely} in {\it full orthogonal form} as
\begin{equation*}
    a=\sum_{i=1}^n t_i e_i, \text{ where the }t_i\in R\text{ are distinct and }e_i\in B.
\end{equation*}

\begin{definition}[{\cite[Definition 2.3]{BMMO}}] An $R$-algebra $A$ is a {\it Specker algebra} if there is a Boolean subalgebra $B$ of $\Id(A)$ whose nonzero elements are faithful and generate $A$ as an $R$-algebra.
\end{definition} 

An $R$-algebra $A$ is a Specker algebra if and only if $A$ is isomorphic to $R[B]$ for some Boolean algebra $B$ \cite[Theorem 2.7]{BMMO}. 
If $R$ is a domain, then a Specker algebra is simply a torsion-free idempotent generated $R$-algebra. For these and a number of other characterizations of Specker algebras, see \cite[Section 2]{BMMO}. 

 \section{Patch spaces, patch presheaves and patch bundles}\label{sec: 3 patch spaces, presheaves, bundles}

Our focus throughout the paper is on collections of subrings of a fixed ring $R$. 
The set $\Sigma(R)$ of all subrings of a ring $R$ admits a natural topology, the {\it Zariski topology}, which has as a basis of open sets the sets of the form $${\mathcal{U}}(r_1,\ldots,r_n) := \{S \in \Sigma(R):r_1,\ldots,r_n \in S\}.$$
We write $\Sigma(R)^{\rm zar}$ when we wish to work with the Zariski topology on  $\Sigma(R)$. The space $\Sigma(R)^{\rm zar}$
is a {\it spectral space}, that is, $\Sigma(R)^{\rm zar}$ is a quasicompact\footnote{To align with the usual terminology  of commutative algebra and algebraic geometry, we say  that a space  in which every open cover has a finite subcover is {\it quasicompact}.  We reserve the term ``compact'' for settings in which we have also the Hausdorff property, but even in those cases, to align with the terminology from general topology, we will be explicit and write ``compact Hausdorff'' rather than simple ``compact.''} $T_0$ space for which the collection of basic open sets  ${\mathcal{U}}(r_1,\ldots,r_n)$ is closed under finite intersections, and every nonempty irreducible closed subset of $\Sigma(R)^{\rm zar}$ has a  generic point \cite[II.3.4]{Joh86}. We will refer to a subset $X$ of $\Sigma(R)^{\rm zar}$ equipped with the subspace topology induced by the Zariski topology as a subspace of (the topological space) $\Sigma(R)^{\rm zar}$.

\subsection{Patch spaces}\label{subsection: patch spaces}

Besides the Zariski topology, we can endow the set $\Sigma(R)$ of subrings of a ring $R$ with a finer topology, the {\it patch topology}. The patch topology on $\Sigma(R)$ has as its basis the sets in $\Sigma(R)^{\rm zar}$ that are an intersection of a quasicompact open set and a complement of a quasicompact open set. A quasicompact open set in $\Sigma(R)^{\rm zar}$ is a finite union of sets of the form ${\cal U}(r_1,\ldots,r_n)$, where $r_1,\ldots,r_n \in R$, and a complement of such a set is a finite union of sets of the form  $${\cal V}(r):=\{S \in \Sigma(R) :r  \not \in S\}, {\mbox{ where }} r \in R.$$ It follows that the patch topology of $\Sigma(R)^{\rm zar}$ has a basis of clopen sets of the form
\begin{equation}\label{eq: clopens}
{\cal U}(r_1,\ldots,r_n)\text{ and }
{\cal U}(r_1,\ldots,r_n) \cap {\cal V}(r)\text{, where }r_1,\ldots,r_n,r \in R.  
\end{equation}
In the patch topology, $\Sigma(R)$ is a Stone space and the designated basis is precisely the set of clopens in \eqref{eq: clopens} (see \cite[Proposition II.4.5]{Joh86}). We denote by ${\Sigma}(R)^{\rm patch}$ the set $\Sigma(R)$ with the patch topology. Thus when we refer to a subspace $X$ of ${\Sigma}(R)^{\rm patch}$ we mean a subset $X$ of $\Sigma(R)^{\rm patch}$ equipped with the subspace topology induced by the patch topology. 

A subspace $X$ of $\Sigma(R)^{\rm zar}$ or of ${\Sigma}(R)^{\rm patch}$ is said to be closed, open or clopen if it exhibits the respective property (closed, open or clopen)~within the subspace topology. An exception to this is that of a spectral subspace. 
Following \cite{ST}, a subspace $X$ of $\Sigma(R)^{\rm zar}$ is a {\it spectral subspace}   if  $X$ is a spectral space and  
the inclusion map $X\hookrightarrow \Sigma(R)^{\rm zar}$ is a spectral map with respect to the Zariski topology, i.e., the preimage of every quasicompact open set of $\Sigma(R)^{\rm zar}$ is quasicompact. (Spectral subspaces are called spectral subobjects in \cite{Hochster}.) Since the quasicompact open subsets of $\Sigma(R)^{\rm zar}$ are finite unions of the sets of the form ${\mathcal{U}}(r_1,\ldots,r_n)$, where $r_1,\ldots,r_n \in R$, a subspace $X$ of $\Sigma(R)^{\rm zar}$ is spectral if and only if the open subsets of $X$  of the form  $X \cap {\mathcal{U}}(r_1,\ldots,r_n)$ are quasicompact.

We specialize to our context and rephrase in our terminology an observation from Hochster \cite[p.~45]{Hochster} that will be useful in the following.

\begin{proposition}\label{prop: Hochster}
A subset of $\Sigma(R)$ is a spectral subspace of $\Sigma(R)^{\rm zar}$ if and only if it is a closed subspace of ${\Sigma}(R)^{\rm patch}$. 
\end{proposition}

The closed subspaces of $\Sigma(R)^{\rm patch}$ will be particularly important in what follows, and so we single them out with the following definition.

\begin{definition}\label{dfn: patch space} 
A {\it patch space} for $R$ is a closed subspace of the space ${\Sigma}(R)^{\rm patch}$ of the subrings of $R$ with the patch topology. 
\end{definition}
A subspace $X$ of ${\Sigma}(R)^{\rm patch}$ is thus considered a patch space if and only if a subring $S$ of $R$ belongs to $X$ whenever the following condition holds: for all
$s_1,\ldots,s_n \in S$, there exists a ring $T$ in $X$ such that $s_1,\ldots,s_n \in T$, and for all 
$s_1,\ldots,s_n \in S$ and $s \in R\setminus S$, there exists a ring $T$ in $X$ such that $s_1,\ldots,s_n \in T$ and $s \in R\setminus T$.

\medskip

Since a subspace of a topological space is dense in its closure, every subspace $X$ of ${\Sigma}(R)^{\rm patch}$ lies dense in a patch space: simply take the patch closure of $X$ in ${\Sigma}(R)^{\rm patch}$. 
However, patch closures can often be difficult to calculate and the rings that are limit points of $X$ may not preserve properties that rings in $X$ have in common. There are, though, important examples of patch spaces of rings, a few of which we list in Example~\ref{example: patch spaces} in order to motivate the concept. Approaches such as in \cite{Fin}, \cite{FFS} and \cite{OIrr} suggest that there will be many variations on these examples. 

\begin{example} \label{example: patch spaces}
Let $R$ be a ring, and let $S$ be a subring of $R$. 

\begin{enumerate}[label=\textup{(\arabic{*})}, mode=unboxed]
\item  (Finocchiaro \cite[Proposition~3.5]{Fin}) The set of all rings between $S$ and $R$ is a patch space for $R$.

\item (Finocchiaro \cite[Proposition 3.6]{Fin})  The set of all rings between  $R$ and $S$ that are integrally closed in $R$ is a patch space for $R$. 

\item (Finocchiaro-Fontana-Spirito \cite[Corollary~2.14]{FFS}) If $R$ is a field, then the set of all local rings between $S$ and $R$ is a patch space for $R$.

\item  If $R$ is a field and $S$ is a local subring of $R$, then the set of all local rings between $S$ and $R$ that dominate $S$ is a patch space for $R$. This is because the map $\phi$ from the space of local rings between $S$ and $R$ to $\Spec(S)$ that sends each local ring  to the contraction of its maximal in $S$ is a continuous map, and so the fiber of this map over the maximal ideal $S$ (which is a  closed point in the Zariski topology of $\Spec(S)$) is a Zariski-closed, hence patch-closed, subset of the patch space of local rings between $S$ and $R$. Thus the set of local rings between $S$ and $R$ dominating $S$ is a patch space. 

\item If $R$ is a field, then the set of all valuation rings between $S$ and $R$ with the Zariski topology is a patch space for $R$, known as the  {\it Zariski-Riemann space} of the extension $R/S$. That this is so has been proved by several authors; see \cite[Example 2.2(8)]{OIrr} for a discussion.  

\item Suppose $R$ is a field, $S$ is a Noetherian subring of $R$ and  $t_1,\ldots,t_n$ are nonzero elements of $R$. For each $i$, define $S_i = S[t_1/t_i,\ldots,t_n/t_i]$ and let $X_i$ be the set of localizations of $S_i$ at its prime ideals.  Each $X_i$ is a patch space. This follows from Proposition~\ref{prop: Hochster} because $X_i$  is spectral (see \cite[Proposition~3.1]{OZar}) and, as follows from the discussion of spectral subspaces above, the inclusion map of $X_i$ into the spectral space consisting  of the local subrings of $R$ is a spectral map since the assumption that $S$, and hence $S_i$, is Noetherian, implies $X_i$ is a Noetherian space and hence  every open set in $X_i$ is quasicompact. Thus $X=X_1 \cup \cdots \cup X_n$ is closed in ${\Sigma}(R)^{\rm patch}$. The space $X$ is the set of stalks of the projective integral scheme over $\Spec(S)$ that is defined by the $t_1,\ldots,t_n$  (cf.~\cite[Remark~2.1]{OGeom}).  More generally, every projective integral scheme arises this way and hence the set of stalks of such a scheme, when viewed as subrings of the function field of the scheme, is a patch space for the function field. 
\end{enumerate}
\end{example}

\subsection{Patch presheaves}\label{subsec: patch presheaves}

It will be useful in later sections to consider intersections of rings in  special subspaces $X$ of $\Sigma(R)^{\rm zar}$. We achieve this by employing a natural choice of presheaf on $X$, which is defined subsequently. Additionally, we introduce a collection of these objects, denoted as ${\cal R}(X)$, serving as an example of a patch presheaf (refer to Definition~\ref{dfn: patch presheaf}). The algebraic properties of ${\cal R}(X)$ will play a pivotal role in the next sections.

\begin{definition}\label{dfn: R(X)} Let $X$ be a subspace of ${\Sigma}(R)^{\rm patch}$. Define a map $\O$ from the set of (patch) open sets of $X$ to ${\Sigma}(R)^{\rm patch}$ by assigning $\O_{X}(\emptyset) = R$, and for each nonempty open set $U$ of $X$, assigning $$\O_{X}(U) = \bigcap_{S\in U}S.$$ 
The collection of all such intersections over clopen sets of $X$ is denoted by $$
{\mathcal R}(X)  :=  \{\O_X(U):U {\mbox{ is clopen in }X}\}.$$   
\end{definition}

Since any subspace of a zero-dimensional topological space is zero-dimensional, the subspace $X$ of ${\Sigma}(R)^{\rm patch}$ has a basis of clopen sets. These clopen sets form a Boolean algebra in the natural way, with union as join and intersection as meet. Thus ${\cal R}(X)$ is indexed by a Boolean algebra, and $\O_X(U \cup V) = \O_X(U) \cap \O_X(V)$ for all clopen sets $U$ and $V$ in $X$.  Abstracting these properties from ${\cal R}(X)$  suggests the following definition.

\begin{definition}\label{dfn: patch presheaf} A {\it patch presheaf} $({\cal R}, B)$ for a ring $R$ is a collection ${\cal R}=\{R_e:e \in B\}$  of subrings of $R$ indexed by a Boolean algebra $B$ such that $R=R_0$ (with $0$ the bottom element of $B$), $R_e \cap R_f = R_{e \vee f}$ for all $e,f \in B$, and so also $R_f \subseteq R_e$ whenever $e \leq f$. \end{definition}

The reason for  the adjective ``patch'' here, in spite of the fact that there is not a patch-closed subset of $\Sigma(R)$ that plays a role in the definition, will be clearer in the next subsection, after the notion of patch bundle is introduced. 
 The reason for the terminology ``presheaf''  is that  ${\cal R}$ can be viewed as a presheaf on the Stone space $\Max(B)$ of the Boolean algebra $B$. However, for the sake of simplifying the algebraic presentation later, we choose to regard ${\cal R}$ as a ``presheaf on $B$'' rather than $\Max(B)$. More precisely, we represent the rings of sections of ${\cal R}$ by indexed subrings $R_e$ of $R$. Each such ring $R_e$ corresponds to the ring of sections of a presheaf defined on the clopen set $\{\m \in \Max(B):e \not \in \m\}$.  This representation can be extended from clopen sets of $\Max(B)$ to all open sets of $\Max(B)$ in the standard way by defining the ring of sections over an arbitrary open set $U$ as the inverse limit of the $R_e$, for each $e \in B$ such that  $\{\m \in \Max(B):e \not \in \m\} $ is a clopen set of $\Max(B)$ contained in $U$. Through this process, ${\cal R}$ can be reinterpreted as a presheaf on $\Max(B)$, hence the terminology ``patch presheaf''.

\begin{remark} \label{orthogonal is enough}
An {\it a priori} weaker condition suffices to define a patch presheaf. If in the definition we replace ``$R_e \cap R_f = R_{e \vee f}$ for all $e,f \in B$'' with the requirement ``$R_{e} \cap R_f = R_{e \vee f}$ for all orthogonal pairs $e,f \in B$'', then  $({\cal R}, B)$ is still a patch presheaf. To see this, suppose $e,f \in B$ and let $\bar{e},\bar{f}\in B$ be their complements (i.e., $e\vee \bar{e}=f\vee \bar{f}=1$ and $e\wedge \bar{e}=f\wedge\bar{f}=0$). Since $B$ is a Boolean algebra, $\bar{e}$ and $\bar{f}$ always exist and we can write $$e \vee f = e' \vee (e \wedge f) \vee f',$$ where $e'=e\wedge \bar{f}$ and $f'=f\wedge \bar{e}$ satisfy the following identities:
\begin{center} $ e' \vee (e \wedge f) =e$, $e' \wedge (e \wedge f) = 0$,    $ f' \vee (e \wedge f) =f$, $f' \wedge (e \wedge f) = 0$.
\end{center}
Since $e',e \wedge f, f'$ are pairwise orthogonal elements of $B$, we have by assumption that $$R_{e \vee f} = R_{ e' \vee (e \wedge f) \vee f'} = R_{e'} \cap R_{e \wedge f} \cap R_{f'} = R_{e' \vee (e \wedge f)} \cap R_{(e \wedge f)\vee f'}=R_e \cap R_f.$$
Thus, when verifying whether $({\cal R}, B)$ is a patch presheaf, it is enough to consider orthogonal elements of $B$.
 \end{remark} 

 \begin{example}\label{ex: patch presheaf}
 In the notation of Definition~\ref{dfn: R(X)}, if $X$ is a subspace of ${\Sigma}(R)^{\rm patch}$,  $$({\cal R}(X), {\rm Clop}(X))$$ is a patch presheaf for $R$.  In fact, the zero element of the Boolean algebra of clopens of $X$ is the empty set, and $\O_X(\emptyset) = R$.
 \end{example}

\begin{definition}\label{dfn: stalks} The {\it stalks} of a patch presheaf $({\cal R}, B)$ for $R$ are the rings ${\cal R}_{\m} = \bigcup_{e \not \in {\ff m}} R_e$, where ${\mathfrak{m}}$ is a maximal ideal in $B$. The collection of all stalks is denoted by 
$${X}({\mathcal R})  :=   \{{\cal R}_{\m}:\m \in \Max(B)\}.$$
\end{definition}
The fact that ${\cal R}_\m$ is a ring for every maximal ideal $\m$ of $B$ follows from the fact that, for every $e,f\in B$, $R_e$ and $R_f$ are subrings of $R_{e\wedge f}$ and $e\wedge f\not\in \m$ if and only if $e\not \in \m$ and $f\not \in \m$ (see Proposition~\ref{prop: max ideals}\ref{prop: max ideals 2}).

The next lemma shows that the rings $R_e$ in a patch presheaf can be recovered from their stalks.

\begin{lemma}\label{stalks lemma} Let $({\cal R}, B)$ be a patch presheaf for a ring $R$. Then, for every $e\in B$, \[R_e = \bigcap_{ {\mathfrak{m}} \not \ni e}{\cal R}_\m,\] where ${\mathfrak{m}}$ ranges over the maximal ideals of $B$ that do not contain $e$.
\end{lemma}

\begin{proof} 
It is clear that $R_e \subseteq \bigcap_{ {\mathfrak{m}} \not \ni e}{\cal R}_\m$. To prove the reverse inclusion, let $t \in \bigcap_{ {\mathfrak{m}} \not \ni e}{\cal R}_\m$.
For each ${\mathfrak{m}}$ not containing $e$, there is $f_{\mathfrak{m}} \in B \setminus {\mathfrak{m}}$ such that $t \in R_{f_{\mathfrak{m}}}$. Consider the ideal $I$ of $B$ generated by the set of elements  $e \wedge f_{\mathfrak{m}}$ such that ${\mathfrak{m}}$ does not contain $e$. We claim $I = eB$. 
Since $B$ is a Boolean algebra, it suffices to show that $I$ and $eB$ are contained in the same maximal ideals of $B$ (see Proposition~\ref{prop: max ideals}\ref{prop: max ideals 4}).  
If ${\mathfrak{n}}$ is a maximal ideal  of $B$ not containing $e$, then $f_{\mathfrak{n}} \not \in {\mathfrak{n}}$ so that $I \not \subseteq {\mathfrak{n}}$.  
On the other hand, if ${\mathfrak{n}}$ is a maximal ideal  of $B$ not containing $I$, then there is a maximal ideal ${\mathfrak{m}}$ not containing $e$ such that $e \wedge f_{\mathfrak{m}} \not \in {\mathfrak{n}}$, and hence $e \not \in {\mathfrak{n}}.$ Therefore, $I = eB$. It then follows from Proposition~\ref{prop: generated ideals}, that there are maximal ideals ${\mathfrak{m}}_1,\ldots,{\mathfrak{m}}_n$ in $B$ that do not contain $e$ and for which $$e \le (e \wedge f_{{\mathfrak{m}}_1}) \vee \cdots \vee (e \wedge f_{{\mathfrak{m}}_n}).$$ Since $({\cal R}, B)$ is a patch presheaf, $$ R_{e \wedge f_{{\mathfrak{m}}_1}} \cap  \cdots \cap R_{e \wedge f_{{\mathfrak{m}}_n}}\subseteq R_e,$$
and so since $t \in R_{f_{{\mathfrak{m}}_i}} \subseteq R_{e \wedge f_{{\mathfrak{m}}_i}}$ for each $i$, we conclude $t \in R_e$, which proves the lemma. 
\end{proof} 

We conclude this subsection by defining a morphism between two patch presheaves for $R$.

\begin{definition}\label{dfn: patch presheaf morphism}
Let $({\cal R}_1=\{R_e : e\in B_1\}, B_1)$ and $({\cal R}_2=\{S_f : f\in B_2\}, B_2)$ be two patch presheaves for the same ring $R$. A {\it morphism of patch presheaves} from $({\cal R}_1, B_1)$ to $({\cal R}_2, B_2)$ consists of a morphism of Boolean algebras $h: B_2\rightarrow B_1$, together with inclusions of rings $S_f\subseteq  R_{h(f)}$, with $f\in B_2$. Note that for every $f'\le f\in B_2$ the following diagram commutes:
\begin{center}
\begin{tikzcd}
S_{f'} \arrow[r, "\subseteq"]
& R_{h(f')} \\
S_f \arrow[u, "\subseteq"] \arrow[r, "\subseteq" ]
& \arrow["\subseteq", u] R_{h(f)} 
\end{tikzcd}
\end{center}
\end{definition} 

With the above definition of morphism, it is easy to check that patch presheaves for a ring $R$ form a category. We will denote the category of patch presheaves for a ring $R$ as {\bf PPresheaf}$_R$, and we will further discuss it in Section~\ref{sec: 4 A category equivalence}.

\subsection{Patch bundles}\label{subsec: patch bundles}

 One of our main goals is to study topologically interesting collections of subrings of a ring $R$. This includes patch spaces of subrings, as defined in subsection~\ref{subsection: patch spaces}, which are Stone spaces. Our methods work  
 also for  a more general class of subspaces of subrings of $R$, namely, those spaces $X$  that are a continuous image of a Stone space $Y$. Thus, if $f:Y \rightarrow X$ is a continuous map, $f$ can be viewed as introducing a Stone space into the picture when there is  not one present, and it allows us to work not just with patch spaces but quotient spaces of Stone spaces. 

To express this we use the language of bundles, where, following \cite[p.~170]{Joh86}, a {\it bundle} over a topological space $X$ is a continuous map $f:Y \rightarrow X$ of topological spaces.

\begin{definition} 
 If $R$ is a ring and $f:Y \rightarrow \Sigma(R)^{\rm zar}$ is a continuous map from a Stone space~$Y$ to the space $\Sigma(R)^{\rm zar}$ of subrings of $R$ (with the Zariski topology), then $f$ is said to be a {\it patch bundle over $\Sigma(R)^{\rm zar}$} or, for short, a {\it patch bundle over $R$}. We define a {\it morphism of patch bundles} from $f_1:Y_1 \rightarrow \Sigma(R)^{\rm zar}$ to $f_2:Y_2 \rightarrow \Sigma(R)^{\rm zar}$ as a continuous map $g:Y_1 \rightarrow Y_2$ such that $(f_2 \circ g) (y_1)\subseteq f_1(y_1)$ for every $y_1\in Y_1$. With this definition of morphism, patch bundles over a ring $R$ form a category that we will denote as ${\bf PBundle}_R$.
\end{definition}

Each object $f: Y\rightarrow \Sigma(R)^{\rm zar}$ in ${\bf PBundle}_R$ picks out both a Stone space $Y$ as well as a space $X$ of subrings of $R$ (the image of $Y$ through $f$), for which $Y$ works as a parameterization space. With a slight abuse of notation we will sometimes identify $f$ with the surjective (continuous) map $Y \rightarrow X$. By Stone duality for Boolean algebras, $Y$ can enrich the Boolean algebra of clopens of $X$, which may be very few, with its own abundance of clopens, as $Y$ is a Stone space.  The need for such an enrichment arises from the fact that $X$ may not be patch-closed in ${\Sigma}(R)$, and hence $X$ will be less directly amenable to  the algebraic constructions of the patch algebras defined later.\footnote{Although not necessary for what follows, a geometric explanation for this move from $X$ to a patch bundle $Y \rightarrow X$ is that the ring of global sections of the sheafification of the patch presheaf induced by the bundle will be the patch algebra of Section~\ref{sec: 6 Patch algebras}. This point of view will be explored in a sequel to this paper.} 

As an image of a Stone space under a continuous function, the image of a patch bundle is quasicompact. However, $X$ need not be a Stone space, as a quotient space of a Stone space is not necessarily a Stone space. This is clear also in light of the following proposition.

\begin{proposition} \label{prop: KHaus} 
A subspace of $\Sigma(R)^{\rm zar}$ is the image of a patch bundle if and only if it is the image of a compact Hausdorff space under a continuous function. 
\end{proposition}

\begin{proof} 
This follows from the  fact that every compact Hausdorff space $X$ is the continuous image of a Stone space.  Specifically, if $X_{disc}$ is the set $X$ with the discrete topology, then the identity map $X_{disc} \rightarrow X$ is continuous, and by the universal mapping property of the Stone-\v{C}ech compactification, this map lifts to a continuous map from the Stone-\v{C}ech compactification $\beta X_{disc}$ of $X_{disc}$ to $X$. Since $X_{disc}$ is discrete, $\beta X_{disc}$ is  a Stone space; see for example  \cite[\href{https://stacks.math.columbia.edu/tag/090C}{Tag 090C}]{stacks-project}. 
\end{proof}

We construct next an example of a patch bundle whose image is not a patch space, thus showing that patch bundles are a more general concept than that of patch spaces. To do so, we first observe in the next proposition that there is nothing special about the spectral spaces that can be realized as spaces of subrings of a ring.

\begin{proposition} \label{not special} For every spectral space $X$, there is a domain $D$  such that $X$ is homeomorphic to a space of rings (with the Zariski topology) between $D$ and its quotient field.  
\end{proposition} 

\begin{proof} Let $X$ be a spectral space. By Hochster's theorem \cite{Hochster}, there is a commutative ring $A$ such that $X$ is homeomorphic to $\Spec(A)$.  Let $D$ be a polynomial ring in possibly infinitely many variables over the ring of integers such that  $A$ is a homomorphic image of $D$. Let $F$ be the quotient field of $D$.  Since $A$ is an image of $D$, $X$ is homeomorphic to a closed set $X'$ of $\Spec(D)$, and so the subspace $\{D_P:P \in X'\}$  of $\Sigma(F)^{\rm zar}$, which is homeomorphic to $X'$ \cite[Lemma~1, p.~116]{ZS},  is therefore homeomorphic to $X$. 
\end{proof}

\begin{example} Let $X$ be a spectral space, and let $Y$ be a subspace of $X$ that is a spectral space but not patch-closed in $X$. Such an example can be found in \cite[Example 2.1.2, p.~49]{ST}. By Proposition~\ref{not special}, we can assume $Y$ and $X$ are subspaces of $\Sigma(F)^{\rm zar}$ for some field $F$.  Since $Y$ is not patch-closed in $X$, $Y$ is not patch-closed in the larger space $\Sigma(F)^{\rm zar}$. Let $Y_P$ denote $Y$ with the patch topology. Since $Y$ is spectral, $Y_P$ is a Stone space. Also, the identity map $Y_P \rightarrow Y$ is continuous since the patch topology is finer than the Zariski topology, so this map is a patch bundle whose image is not a patch space. 
\end{example}

\begin{example} One way to construct patch bundles whose image is a patch space is from homomorphisms of the Boolean algebra of clopens of $\Sigma(R)^{\rm patch}$ into arbitrary Boolean algebras.  
If $\alpha:{\rm Clop}({\Sigma}(R)^{\rm patch}) \rightarrow B$ is any homomorphism of Boolean algebras, then the map $\Max(B) \rightarrow {\Sigma}(R)^{\rm patch}$ that is defined as the composition of the 
dual map $$\Max(B) \rightarrow \Max({\rm Clop}({\Sigma}(R)^{\rm patch}))$$ with the homeomorphism $$\Max({\rm Clop}({\Sigma}(R)^{\rm patch})) \rightarrow {\Sigma}(R)^{\rm patch}$$ is a patch bundle with the additional property that the mapping defining the bundle is continuous with respect to the patch topology on $\Sigma(R)$, not just the Zariski topology. Since $\Max(B)$ is quasicompact and $\Sigma(R)^{\rm patch}$ is Hausdorff, the Closed Map Lemma implies the image of the bundle $\Max(B) \rightarrow \Sigma(R)^{\rm patch}$ is closed, and hence is a patch space.   
\end{example}

It would be interesting however to have a more transparent description of the subspaces of $\Sigma(R)^{\rm zar}$ that are the images of patch bundles, or,  equivalently, images of compact Hausdorff spaces (Proposition~\ref{prop: KHaus}).

\section{A category equivalence}\label{sec: 4 A category equivalence}

The dichotomy of patch bundles vs.~patch presheaves 
is essentially  that of  stalks vs.~rings of sections for ringed spaces, and just as in some cases it is easier to work with stalks in \'{e}tale spaces rather than rings of sections over open sets, or vice versa, the concepts of patch bundle and patch presheaf afford a similar flexibility. To make this relationship precise, we show in this section that the categories {\bf PBundle}$_R$ of patch bundles and {\bf PPresheaf}$_R$ of patch presheaves over the same ring $R$ are equivalent. Observe first that 
each patch bundle over $\Sigma(R)^{\rm zar}$ gives rise to a patch presheaf in a natural way:

\begin{definition}\label{dfn: patch presheaf of a patch bundle} Let $f:Y \rightarrow \Sigma(R)^{\rm zar}$ be a patch bundle, so that $f(Y)$ is a collection of subrings of $R$.  For each nonempty $U \in $ Clop$(Y)$, define 
\begin{center} $R_U = $  the intersection of the rings in $f(U),$  
\end{center}
where $R_\emptyset = R$. The  
 {\it patch presheaf for} $f$ is defined by 
\begin{center}
${\bf R}(f) = \left(
\{R_U:U \in {\rm Clop}(Y)\},  {\rm Clop}(Y)\right)$.
\end{center}
\end{definition}

This assignment of a patch presheaf via ${\bf R}$ to a patch bundle is functorial, as we will show, and it is part of a category equivalence. To prove this we require two lemmas, the first of which shows that ${\bf R}(f)$ encodes the rings in the image of a patch bundle $f$ as the stalks of the presheaf ${\bf R}(f)$.

\begin{lemma} \label{lem: bundle and stalks}
The image of a patch bundle $f:Y \rightarrow \Sigma(R)^{\rm zar}$  is the set of stalks of the patch presheaf ${\bf R}(f)$.    
\end{lemma} 

\begin{proof}  
Let $f:Y \rightarrow \Sigma(R)^{\rm zar}$ be a patch bundle and let $X:=f(Y)$ be the  image of $Y$. Let $S \in X$, and choose $y \in Y$ such that $f(y) = S$. By Theorem~\ref{thm: Stone's repr. thm}, $\m=\{U\in {\rm Clop}(Y):y \not \in U\}$ is a maximal ideal of Clop$(Y)$ and the stalk $\bigcup_{U \not \in \m} R_U$  of the patch presheaf ${\bf R}(f)$ at $\m$ satisfies the following inclusion: 
$$ \bigcup_{U \not \in \m} R_U = \bigcup_{U\ni y} R_U \subseteq f(y).$$
In fact, for every $U\in {\rm Clop}(Y)$ containing $y$, $S=f(y)\in f(U)$, and so $$R_U=\bigcap_{T\in f(U)}T\subseteq S.$$
On the other hand, let $t \in f(y)$. Then $f(y) \in {\cal U}(t)$, where $$\mathcal{U}(t)=\{T\in \Sigma(R)^{\rm zar} \, \colon \, t\in T\}$$ is open in $\Sigma(R)^{\rm zar}$. So $y \in f^{-1}({\cal U}(t))$. Since $Y$ is a Stone space and $f^{-1}({\cal U}(t))$ is an open neighborhood of $y$ in  $Y$, there is a clopen $U$ with $y \in U \subseteq f^{-1}({\cal U}(t))$, and so $f(y) \in f(U) \subseteq {\cal U}(t)$. Thus $t \in \bigcup_{U\ni y}R_U$, which proves $S=f(y)$ is the stalk of ${\bf R}(f)$ at $\m$, and so $X\subseteq X({\bf R}(f))$. On the other hand, let $S$ be the stalk of ${\bf R}(f)$ at a maximal ideal $\m$ of ${\rm Clop}(Y)$, i.e., $S\in X({\bf R}(f))$. By Stone duality, $\m=\{U\in {\rm Clop}(Y):y \not \in U\}$ for some $y\in Y$. The above argument shows then that $S=f(y)\in X$, proving that $X = X({\bf R}(f))$. \end{proof}    

Before proving the next lemma, let us recall from \cite[II.3.5]{Joh86} and Section~\ref{subsec: Stone spaces} that for any Boolean algebra $B$ the maximal spectrum $\Max(B)$ of $B$, whose points are the maximal ideals of $B$, is a Stone space under the patch topology or, equivalently, the Zariski topology. The clopen sets of $\Max(B)$ are of the form $\{\m \in \Max(B): e \not \in \m\}$, where $e\in B$. 

\begin{lemma}\label{lem: continous surjection} If $({\cal R}, B)$ is a patch presheaf for a ring $R$, then the map  $$\phi:\Max(B) \rightarrow X({\mathcal R}):{\mathfrak{m}} \mapsto {\mathcal R}_\m$$ is a surjection that is continuous in the Zariski topology and thus is a patch bundle. 
\end{lemma} 

\begin{proof} It is clear that the map is surjective. To see that it is continuous,
let $t_1,\ldots,t_n \in R$. Using the fact that ${\mathcal{R}}_{\m}$ is a directed union of the $R_e$ with $e \not \in \m$, we have 
\begin{eqnarray*}
\phi^{-1}(X({\cal R}) \cap {\cal U}(t_1,\ldots,t_n) ) & = & 
\{\m \in \Max(B):t_1,\ldots,t_n \in {\cal R}_{\m}\} \\
& = & 
\{\m \in \Max(B):t_1,\ldots,t_n \in R_e {\mbox{ for some }} e \not \in \m\} \\
& = & 
\bigcup_{e\in E} 
\{\m \in \Max(B): e \not \in \m\},
\end{eqnarray*}
where $E:=\{e\in B \, \colon \,  t_1,\dots, t_n\in R_e\}$. As a union of open sets, this last set is open, and so $$\phi^{-1}(X({\cal R}) \cap {\cal U}(t_1,\ldots,t_n))$$ is open, which proves that $\phi$ is continuous in the Zariski topology.  
\end{proof}

In light of the lemma, we can then associate a patch bundle to each patch presheaf. 

\begin{definition}\label{dfn: patch bundle of a presheaf}
Let $({\cal R},B)$ be a patch presheaf for  $R$.  The {\it patch bundle} of  $({\cal R},B)$, denoted ${\bf B}({\cal R},B)$,  is the bundle $\phi:\Max(B) \rightarrow X({\cal R})$ of Lemma~\ref{lem: continous surjection}.  
\end{definition}

Let ${\bf PPresheaf}_R$ be the category of patch presheaves for $R$, where the morphisms are those defined in Definition~\ref{dfn: patch presheaf morphism}. 
We extend the definition of ${\bf B}$ on objects of ${\bf PPresheaf}_R$ to that of a functor
$${\bf B}:{\bf PPresheaf}_R \rightarrow {\bf PBundle}_R.$$  This functor sends a morphism $$({\cal R}_1=\{R_e : e\in B_1\},B_1) \rightarrow ({\cal R}_2=\{S_f : f\in B_2\},B_2)$$ of patch presheaves to a morphism of patch bundles $\phi_1\to \phi_2$, where $\phi_i={\bf B}({\cal R}_i,B_i)$, for $i=1,2$. Such morphism of patch bundles is defined by the dual map $h^*: \Max(B_1) \rightarrow \Max(B_2)$ of the Boolean algebra morphism $h: B_2 \to B_1$ specified by the morphism of patch presheaves, inducing the inclusion $S_f\subseteq R_{h(f)}$ for every $f\in B_2$. 
In fact, \[\phi_2(h^*(\m_1))=\phi_2(h^{-1}(\m_1))=(\cal{R}_2)_{h^{-1}(\m_1)}=\bigcup_{f\not \in h^{-1}(\m_1)}S_f,\] while \[\phi_1(\m_1)=\bigcup_{e\not \in \m_1}R_e.\] Thus, if $t\in \phi_2(h^*(\m_1))$, then $t\in S_f\subseteq R_{h(f)}$ for some $f\not\in h^{-1}(\m_1)$. Since $h(f)\not \in \m_1$, then $t\in \phi_1(\m_1)$, which verifies that $\phi_2(h^*(\m_1)) \subseteq \phi_1(\m_1)$ and hence that $h^*$ gives a morphism of the patch bundles $\phi_1$ and $\phi_2$.

We also extend ${\bf R}$, already defined on objects, to a functor $${\bf R}: {\bf PBundle}_R  \rightarrow {\bf PPresheaf}_R.$$
To define ${\bf R}$ on morphisms in ${\bf PBundle}_R$, let $g:Y_1 \rightarrow Y_2$ be a continuous map defining a morphism between the patch bundles $$
f_1:Y_1  \rightarrow X_1 \: {\mbox{ and }} \: f_2 : Y_2  \rightarrow X_2.
$$
Namely, $f_2(g(y_1))\subseteq f_1(y_1)$ for every $y_1\in Y_1$. Set 
\begin{eqnarray*}
{\bf R}(f_1) & = & \left(
\{R_U:U \in {\rm Clop}(Y_1)\},  {\rm Clop}(Y_1)\right) \\ {\bf R}(f_2) & = & \left(
\{S_V:V \in {\rm Clop}(Y_2)\},  {\rm Clop}(Y_2)\right).
\end{eqnarray*}
The dual map of $g$ given by
$$g^*:{\rm Clop}(Y_2) \rightarrow {\rm Clop}(Y_1):V \mapsto g^{-1}(V)$$ is a homomorphism of Boolean algebras. We claim that, for each $V \in {\rm Clop}(Y_2)$, $S_V\subseteq R_{g^*(V)}=R_{g^{-1}(V)}$, i.e., that $g^*$ defines a morphism of patch presheaves from ${\bf R}(f_1)$ to ${\bf R}(f_2)$. This follows from Lemmas~\ref{stalks lemma} and~\ref{lem: bundle and stalks}, the defining property of $g$ and the observation that $\{v\in V: v\in g(U)\}\subseteq V$:
\[ S_V=\bigcap_{v\in V}f_2(v)\subseteq \bigcap_{u\in g^{-1}(V)} f_2(g(u))\subseteq \bigcap_{u\in g^{-1}(V)} f_1(u)=R_{g^*(V)}.\]

It is routine to check that both ${\bf B}$ and ${\bf R}$ preserve composition of morphisms.

\begin{theorem}\label{thm: cat equiv} The functors ${\bf R}$ and ${\bf B}$ define an equivalence between the
category ${\bf PPresheaf}_R$ of patch presheaves for $R$ and the category ${\bf PBundle}_R$ of patch bundles over $R$.  
\end{theorem} 

\begin{proof} 
Let $f:Y \rightarrow X$ be a patch bundle. Then, by Lemma \ref{lem: bundle and stalks}, ${\bf B}({\bf R}(f))$ is the patch bundle $$\Max({\rm Clop}(Y))  \rightarrow X.$$
Similarly, if $f_1\to f_2$ is a morphism between the patch bundles $f_1:Y_1\rightarrow X_1$ and $f_2:Y_2\rightarrow X_2$, then 
${\bf B}({\bf R}(f_1\to f_2))$ is the morphism from the patch bundle ${\bf B}({\bf R}(f_1))$ to the patch bundle ${\bf B}({\bf R}(f_2))$ defined by the continuous map $$\Max({\rm Clop}(Y_1)) \rightarrow \Max({\rm Clop}(Y_2)).$$ 
Using that Stone duality is a dual equivalence of categories, we then obtain a natural isomorphism of ${\bf B} \circ {\bf R}$ with the identity functor on ${\bf PBundle}_R$.

Now let $({\cal R},B)$ be a patch presheaf for $R$.  Then ${\bf R}({\bf B}({\cal R},B))$ is the patch presheaf $$(\{R_U:U \in {\rm Clop}(\Max(B)\},{\rm Clop}(\Max(B))),$$ where, using Lemma~\ref{stalks lemma}, $R_U =R_e$ for $e \in B$ such that $U = \{\ff m: e \not \in \m\}$.  Applying ${\bf R} \circ {\bf B}$ to a  morphism of patch presheaves $({\cal R}_1,B_1) \to ({\cal R}_2,B_2)$ results in a morphism of presheaves involving the Boolean algebra homomorphism $${\rm Clop}(\Max(B_2)) \rightarrow {\rm Clop}(\Max(B_1)).$$ So the fact that ${\bf R} \circ {\bf B}$ is naturally isomorphic to the identity functor on ${\bf PPresheaf}_R$ is a consequence of Stone duality.
\end{proof}

\section{Patch spaces and distinguished patch presheaves}\label{sec: 5 patch spaces and distinguished patch presheaves}

In order to describe the precise relationship between patch spaces and patch presheaves, we need an additional restriction on patch presheaves. 

\begin{definition} \label{dfn: dist}
A patch presheaf $({\cal R}, B)$ of subrings of $R$ is {\it distinguished} if
\begin{enumerate}[label=\textup{(\alph{*})}, mode=unboxed]
\item\label{dfn: dist a} for each $t \in R$, there exists $e_t \in B$ such that $t \in R_e$ if and only if $ e \leq e_t$; and
\item\label{dfn: dist b} the $e_t$ separate points; i.e., for distinct maximal ideals $\m$ and ${\mathfrak{n}}$ of $B$, there is $t \in R$ such that  $e_t$ is in one of $\m$ and ${\mathfrak{n}}$ but not the other.
\end{enumerate}
\end{definition}
Condition \ref{dfn: dist a} of being distinguished guarantees that for each $t \in R$, the patch presheaf contains a unique smallest ring  containing $t$, namely, $R_{e_t}$. 

\medskip

Let us recall from Section~\ref{subsection: patch spaces} that ${\Sigma}(R)^{\rm patch}$ denotes the space $\Sigma(R)$ of subrings of $R$ endowed with the patch topology.

\begin{lemma}\label{lem: dist} Let $({\cal R}, B)$ be a distinguished patch presheaf for the ring $R$.
\begin{enumerate}[label=\textup{(\roman{*})}, mode=unboxed]
\item\label{lem: dist 1} If $\m$ is a maximal ideal of $B$, then $t \in {\cal R}_{\m}$ if and only if {$e_t \notin \m$}.
\item\label{lem: dist 2} If we consider $X({\mathcal R})$ as a subspace of  ${\Sigma}(R)^{\rm patch}$, the map $$\phi:\Max(B) \rightarrow X({\mathcal R}):{\mathfrak{m}} \mapsto {\mathcal{R}}_\m$$ is a homeomorphism.
\end{enumerate}
\end{lemma}

\begin{proof} 
\ref{lem: dist 1} If $e_t\notin \m$, then $t\in R_{e_t}\subseteq {\cal R}_{\m}=\bigcup_{e\not\in \m} R_e$. Conversely, suppose $t \in {\cal R}_{\m}$. Then  $t \in R_e$ for some $e\not \in \m$, and so by the choice of $e_t$, we have $e \leq e_t$. Since $e \not \in \m$, it is also the case that $e_t \not \in \m$. 

\ref{lem: dist 2} It is clear that $\phi$ is onto. To see that $\phi$ is injective, let ${\mathfrak{m}}$ and ${\mathfrak{n}}$ be distinct maximal ideals of ${\cal R}$. Since ${\mathcal R}$ is distinguished, there is $t \in R$ such that $e_t $ is in one of $\m$ and ${\mathfrak{n}}$, say $\m$, and not the other. But then by item~\ref{lem: dist 1}, $t  \in   {\cal R}_{\mathfrak{n}} \setminus {\cal R}_{\m}$, proving $\phi$ is injective.  
To see that $\phi$ is continuous in the patch topology, let $r_1,\ldots,r_n,r\in R$, and let 
$$U = \{{\cal R}_{\m}:\m \in \Max(B), r_1,\ldots,r_n \in {\cal R}_{\m}, r \not \in {\cal R}_{\m}\}.$$   
By \ref{lem: dist 1}, 
$$U = \{{\cal R}_{\m}:\m \in \Max(B), e_{r_1} \not \in \m,\ldots, e_{r_n} \not \in \m,e_r \in \m\}.$$ Thus the preimage of $U$ under $\phi$ is the clopen set $$\phi^{-1}(U) = \{\m \in \Max(B): e_{r_1} \not \in \m,\ldots, e_{r_n} \not \in \m,e_r \in \m\}.$$  If we drop $r$ from these calculations, we obtain a similar conclusion, so it follows that $\phi$ is continuous in the patch topology.  
Since $\Max(B)$ is quasicompact and  $X({\mathcal R})$ (as a subspace of ${\Sigma}(R)^{\rm patch}$) is Hausdorff, the Closed Map Lemma (every continuous map from a quasicompact space to a Hausdorff space is closed) implies that $\phi$ is closed. Therefore, as a continuous closed bijection, $\phi$ is a homeomorphism.
\end{proof}

For the next lemma, recall the notation of Definition~\ref{dfn: R(X)}.

\begin{lemma} \label{lem: distinguished R} If  $X$ is a patch space for the ring $R$, then   $({\cal R}(X),{\rm Clop}(X))$ is a distinguished patch presheaf for $R$.
\end{lemma}

\begin{proof}
We have already observed in Example~\ref{ex: patch presheaf} that $({\cal R}(X), {\rm Clop}(X))$ is a patch presheaf for  $R$. We claim that, if $X$ is a (closed) subspace of ${\Sigma}(R)^{\rm patch}$, it is also distinguished. For property \ref{dfn: dist a}, let $t \in R$,  and consider $${\cal U}_X(t)=\{S \in X:t \in S\}.$$ Since  ${\cal U}(t)$ is clopen in the patch topology, ${\cal U}_X(t)$ is clopen in the subspace topology of $X$. If $V$ is another clopen in $X$ such that $t \in \O_X(V)=\bigcap_{S\in V}S$, then $V \subseteq {\cal U}_X(t)$. Thus  ${\cal U}_X(t)$, viewed as an element of the Boolean algebra of clopens of $X$, behaves as $e_t$ in Definition~\ref{dfn: dist}\ref{dfn: dist a}.

To show that the ${\cal U}_X(t)$ play the role of the $e_t$ in   condition \ref{dfn: dist b} in Definition~\ref{dfn: dist}, suppose $\m_1$ and $\m_2$ are distinct maximal ideals of the Boolean algebra of clopens of $X$. 
We show there is $t \in R$ such that ${\cal U}_X(t)$ is in one of $\m_1$ and $\m_2$ but not the other.
By Theorem~\ref{thm: Stone's repr. thm}, there are rings $S_1,S_2 \in X$ such that 
\begin{center} $\m_i = \{U \in $ Clop$(X):S_i \not \in U\}.$
\end{center}
Since the patch topology is Hausdorff and zero-dimensional, there is a subbasic clopen set $U={\cal U}(r_1,\ldots,r_n) \cap {\cal V}(r)$ of $X$, where $r_1,\ldots,r_n,r\in R$, such that $S_1\in U$ and $S_2\notin U$.  It then follows that there is $t \in \{r_1,\ldots,r_n,r\}\subseteq R$ such that either $S_1 \in {\mathcal{U}}_X(t)$ and $S_2 \not \in {\mathcal{U}}_X(t)$, or 
$S_1 \not \in {\cal U}_X(t)$ and $S_2  \in  {\cal U}_X(t)$.
Either way, exactly one of $S_1$ and $S_2$ is contained in ${\cal U}_X(t)$, and hence ${\cal U}_X(t)$ is in exactly one of $\m_1$ and $\m_2$. Thus the clopens of the form ${\cal U}_X(t) $ play the role of $e_t$ in Definition~\ref{dfn: dist}\ref{dfn: dist b}, and so ${\cal R}(X)$ is distinguished.
\end{proof}

We have stated and proved Lemma \ref{lem: distinguished R} exclusively for patch spaces, a context adequate for the objectives of this section. It is worth observing, nonetheless, that the argument is applicable to any subspace of ${\Sigma}(R)^{\rm patch}$.

\begin{theorem}\label{thm: one-to-one dist patch} 
The mappings $X \mapsto ({\cal R}(X), {\rm Clop}(X))$ and $({\cal R}, B)\mapsto X(\cal R)$ define a one-to-one correspondence between patch spaces for $R$ and distinguished patch presheaves for~$R$.  
\end{theorem} 

\begin{proof}  Let $({\cal R},B)$ be a distinguished patch presheaf for $R$. That $X({\cal R})$ is a patch space for $R$ follows from Lemma~\ref{lem: dist}\ref{lem: dist 2}. Specifically, the lemma implies  $X({\cal R})$ is quasicompact in the patch topology, and as such it is a closed subspace of the compact Hausdorff space $\Sigma(R)^{\rm patch}$. On the other hand, by Lemma~\ref{lem: distinguished R}, $({\cal R}(X),{\rm Clop}(X))$ is a distinguished patch presheaf for every subspace $X$ of the patch space $\Sigma(R)^{\rm patch}$. 

Next, we claim that if $X$ is a patch space for $R$, then $X({\mathcal R}(X)) = X$. 
In showing first that $X \subseteq X({\mathcal R}(X))$ we do not need the assumption that $X$ is a patch space. 
Let $S \in X$. Using Theorem~\ref{thm: Stone's repr. thm} (Stone duality), the set $$\m = \{U \in {\mbox{ Clop}}(X):S \not \in U\}$$ is a maximal ideal of the Boolean algebra Clop$(X)$.
The rings in $X({\cal R}(X))$ are of the form $\O_X(U)$, where $U \in {\rm Clop}(X)$, and so by definition, 
$${\cal R}(X)_{\m} = \bigcup_{U \not \in  \m}\O_X(U).$$ 
Thus 
$$ {\cal R}(X)_{\m} = \bigcup_{U \not \in \m}\O_X(U) = \bigcup_{U\ni S}\O_X(U) =S,$$ where the last equality follows from the fact that if $a \in S$, then the set $X \cap {\cal U}(a)$ is clopen in $X$  and $S \in X \cap {\cal U}(a)$, so that $$a \in
\O_X(X \cap {\cal U}(a)) \subseteq 
\bigcup_{S \in U}\O_X(U)$$ 
(the fact that $\bigcup_{U\ni S}\O_X(U) \subseteq S$ is a consequence of the definition). Thus $S \in X({\cal R}(X))$, which proves $X \subseteq X({\cal R}(X))$. 

For the reverse inclusion, suppose  ${\cal R}(X)_{\m} \in X({\cal R}(X))$, where $\m$ is a maximal ideal of the Boolean algebra of clopens of $X$. By assumption, $X$ is a patch space, so it is Stone and, by Theorem~\ref{thm: Stone's repr. thm}, it is homeomorphic to the Stone space ${\rm Max}({\rm Clop}(X))$. 
Then, there is $S\in X$ such that $$\m = \{U \in {\mbox{Clop}}(X):S \not \in U\}$$ and so $U\not\in \m$ if and only if $S\in U$. As in the previous paragraph, this implies that $$ {\cal R}(X)_\m=\bigcup_{U \not \in  \m}\O_X(U) = \bigcup_{U \ni S}\O_X(U)=
S,$$  which verifies that $X = X({\cal R}(X))$.
 
Finally, for $({\cal R}, B)$ a distinguished patch presheaf for the ring $R$, we claim ${\cal R}(X(\cal R)) = {\cal R}$.  Let $S \in {\cal R}(X(\cal R))$. Then $S = \O_{X({\cal R})}(U)$,  where $U$ is a clopen subset of $X({\cal R})$. The map $$\phi:\Max(B) \rightarrow X({\mathcal R}):{\mathfrak{m}} \mapsto {\mathcal{R}}_\m$$ is a homeomorphism by Lemma~\ref{lem: dist}\ref{lem: dist 2}, so $\phi^{-1}(U)$ is a clopen set of $\Max(B)$, and $$S = 
\bigcap_{\m \in \phi^{-1}(U)}{\cal R}_{\m}.$$ 
Since $\phi^{-1}(U)$ is clopen in $\Max(B)$, there is $e \in B$ such that $$\phi^{-1}(U) =\{\m\in \Max(B):e \not \in \m\}.$$ Therefore, by Lemma~\ref{stalks lemma},  $$S = 
\bigcap_{\m \in \phi^{-1}(U)}{\cal R}_{\m}=
\bigcap_{\m \not \ni e}{\cal R}_{\m} = 
R_e,$$ proving $S \in {\cal R}$. This shows that ${\cal R}(X(\cal R)) \subseteq {\cal R}$. For the reverse inclusion, let $e \in B$. We claim $R_e \in {\cal R}(X(\cal R))$. Using Lemma~\ref{stalks lemma}, $$R_e = \bigcap_{\m \not \ni e}{\cal R}_{\m} =  \O_{X({\cal R})}(\phi(\{\m:e \not \in \m\})),$$
and so $R_e \in {\mathcal R}(X(\cal R))$ since $\phi(\{\m:e \not \in \m\})$ is clopen in $X$. This proves that ${\cal R}(X(\cal R)) = {\cal R}$, which completes the proof of the theorem. 
\end{proof}

\begin{remark}\label{rem: patch spaces are bundles}
    Since every patch space is a Stone space (being closed in ${\Sigma}(R)^{\rm patch}$ with respect to the patch topology), the identity map  of a patch space $X$ is a patch bundle. More precisely, denote by $X_Z$ the space $X$ with the Zariski topology and by $X_P$ the space $X$ with the patch topology. The map $$f:X_P \rightarrow X_Z:x \mapsto x$$ is a patch bundle since every open set in the Zariski topology is open in the patch topology, and hence $f$ is continuous. The associated patch presheaf ${\bf R}(f)$ of the bundle $f$ is $$(\{\O_X(U):U \in {\rm Clop}(X_P)\},{\rm Clop}(X_P)) = ({\mathcal{R}}(X_P),{\rm Clop}(X_P))$$ which is distinguished by Theorem \ref{thm: one-to-one dist patch}.
\end{remark}

\section{Patch algebras} \label{sec: 6 Patch algebras}

  In this section, we consider the problem of realizing a collection of subrings of a given ring $R$ as quotients of a single ring. The pivotal concept in achieving this objective lies in the notion of {\it patch algebra}. The definition of patch algebras builds upon Specker algebras, introduced in \cite{BMO} and discussed in Section \ref{subsec: Specker algebra}.

\begin{definition}\label{dfn: R[presheaf]} The {\it patch algebra of a patch presheaf $({\cal R}, B)$ for $R$}, denoted
${\mathcal{R}}[B]$, is the subring of the Specker $R$-algebra $R[B]$ defined by  
\smallskip
\begin{center}
${\cal R}[B] = \{t_1e_1 + \cdots + t_n e_n : n>0, e_i \in B$ and $t_i \in R_{e_i}$ for every $i\in \{1,\ldots, n\}\}$.
\end{center}
\end{definition} 

The fact that a patch algebra is indeed an algebra over a relevant ring is proved in the next proposition. 

\begin{proposition} \label{prop: R presheaf}
Let $({\cal R},B)$ be a patch presheaf for the ring $R$. 

\begin{enumerate}[label=\textup{(\arabic{*})}, mode=unboxed]
\item\label{prop: R presheaf 1}
${\cal R}[B] $ is an $R_1$-algebra, where $1$ is the top element of $B$. 

\item\label{prop: R presheaf 2} $B$ is a Boolean algebra of faithful idempotents of ${\cal R}[B]$.
\end{enumerate}
\end{proposition}

\begin{proof} 
\ref{prop: R presheaf 1} That  ${\cal R}[B] $ is a ring follows from the fact that if $e,f\in B$, $t\in R_e$, and $s\in R_f$, then since $ef=e \wedge f \le e, f$, we have  $t,s\in R_{ef}$ and so $ts\in R_{ef}$. That ${\cal R}[B] $ is an $R_1$-algebra follows from the fact that $R_1$ is contained in every ring of ${\cal R}$.

\ref{prop: R presheaf 2} As recalled in Section~\ref{subsec: Specker algebra}, the elements of $B$ are faithful idempotents (with respect to the ring $R$) in the Specker algebra $R[B]$. Thus, since $B \subseteq {\cal R}[B]$, $B$ is a Boolean algebra of (faithful) idempotents of ${\cal R}[B]$. 
\end{proof}

We gather in the next lemma some further properties of ${\cal R}[B]$ that will be useful in the rest of the section.

\begin{lemma}\label{lem: R presheaf} Let $({\cal R}, B)$ be a patch presheaf for the ring $R$.  
\begin{enumerate}[label=\textup{(\roman{*})}, mode=unboxed]
    \item\label{lem: R presheaf i} If $a \in {\cal R}[B] $, then $a =  t_1e_1+ \cdots + t_ne_n,$ where $\{e_1,\ldots,e_n\}$ is an orthogonal subset of $B$, and $t_i \in R_{e_i}.$
    \item\label{lem: R presheaf ii} If $\{e_1,\ldots,e_n\}$ is an orthogonal subset of $B$ and $t_1,\ldots,t_n \in R$ are such that $t_1e_1+\cdots +t_ne_n  \in {\cal R}[B] $, then $t_i \in R_{e_i}$ for each $i$. 
\end{enumerate}
\end{lemma}
\begin{proof} \ref{lem: R presheaf i} Since $a \in {\cal R}[B] $, by definition, $a = t_1e_1 + \cdots + t_ne_n,$ where the $e_i$ are elements of $B$ (that we can assume without loss of generality are  different from $0\in B$) and each $t_i \in R_{e_i}$. There are pairwise orthogonal idempotents $f_1,\ldots,f_m$ in $B$ such that each $f_j \leq e_i$ for some $i$. For each $j \in \{1,\ldots,m\}$, let $I_j\subseteq \{1, \dots, n\}$ be the set of $i\in \{1,\dots, n\}$ such that $f_j \leq e_i$, and set $s_j:=\sum_{i\in I_j}t_i$. Since $R_{e_i}\subseteq R_{f_j}$ for every $i\in I_j$, $s_j\in R_{f_j}$ for every $j$. It follows that $a=s_1f_1+\dots + s_m f_m$.

\ref{lem: R presheaf ii}  To ease the notation, set $A := {\cal R}[B] $. It suffices to show that if $te \in A$ for some $t \in R$ and $e \in B$, then $t \in R_e$. Suppose then that, for some $t\in R$ and $e\in B$, $te \in A$. We may assume $t \ne 0$ and $e \ne 0$ since otherwise there is nothing to prove. By item \ref{lem: R presheaf i} there are orthogonal nonzero idempotents $e_1,\ldots,e_n \in B$ and elements $t_1, \ldots, t_n\in R$, with $t_i\in R_{e_i}$, such that $te = t_1e_1 + \cdots + t_n e_n$. In particular, we can assume without loss of generality that $ee_i\ne 0$ for every $i$. Moreover, the fact that the $e_i$ are orthogonal implies $tee_i =t_i ee_i$ for each $i$. Thus $t - t_i \in \ann_R(ee_i)=\{0\}$, and so $t=t_i\in R_{e_i}$ for every $i$. It then follows from faithfulness and the fact that $t\ne0$ that $e=e_1+\cdots+e_n$, so $$t\in R_{e_1} \cap \cdots \cap R_{e_n} = R_{e_1 \vee \cdots \vee e_n} = R_{e_1 + \cdots + e_n}=R_e,$$ which proves item \ref{lem: R presheaf ii}. 
 \end{proof} 
 
\begin{remark}\label{rem: R indecomposable} By item \ref{lem: R presheaf i} of Lemma~\ref{lem: R presheaf}, every $a\in {\cal R}[B]$ can be written in the form $a=t_1e_1 + \cdots + t_n e_n$, for some positive integer $n$, $e_1,\ldots,e_n \in B\setminus\{0\}$ pairwise orthogonal, and $t_i \in R_{e_i}$. If $R$ is indecomposable, since the idempotents of $B$ are faithful, a direct computation shows that $a$ is idempotent only if $a\in B$. Thus, if $R$ is indecomposable, then $B=\Id({\cal R}[B])$ and  every idempotent in ${\cal R}[B]$ is accounted for by $B$.
\end{remark}

\begin{remark}\label{rem: patch presheaf from patch algebra}
  Let $({\cal R},B)$ be a patch presheaf for the ring $R$ and let $A={\cal R}[B]$ be its patch algebra. Then, for every $e\in B$, \[R_e=( A :_R e)=\{ r\in R : re\in A\}.\]
  Indeed, if $r\in R_e$, then $re\in A$ by Definition \ref{dfn: R[presheaf]}. On the other hand, if $re\in A$ for some $r\in R$, then $r\in R_e$ by Lemma~\ref{lem: R presheaf}\ref{lem: R presheaf ii}.
\end{remark} 

We recalled in Section \ref{subsec: Specker algebra} that the elements of the Specker algebra $R[B]$ admit a unique full orthogonal representation with distinct coefficients. The following proposition is a consequence of this uniqueness, together with Lemma~\ref{lem: R presheaf}\ref{lem: R presheaf ii}, which guarantees that the coefficients $t_i$ of an orthogonal decomposition of an element $a\in\cal R[B]$ are in $R_{e_i}$.

\begin{proposition}\label{prop: unique full orth form}
    Let $({\cal R}, B)$ be a patch presheaf for the ring $R$. Every element $a\in {\cal R}[B]$ admits a unique full orthogonal decomposition
    \[ a=\sum_{i=1}^n t_i e_i, \text{ where the }t_i\in R_{e_i}\text{ are distinct and }e_i\in B.\]
\end{proposition}

The next lemma shows that a patch algebra is a solution to the problem of realizing a collection  of subrings of $R$ as factor rings of a single ring. Before proceeding, we observe that in light of the category equivalence established in Theorem \ref{thm: cat equiv} between patch bundles and patch presheaves for a ring $R$, we can define a patch algebra of a patch bundle over $R$ as follows:
\begin{definition}\label{dfn: patch algebra of a patch bundle}
   The {\it patch algebra of a patch bundle $f:Y \rightarrow X$ over $R$} is the patch algebra of the patch presheaf for $R$ given by
   $$({\cal R}, B)={\bf R}(f)=(\{R_U : U\in {\rm Clop}(Y)\}, {\rm Clop}(Y)).$$ Here ${\bf R}$ denotes the functor introduced in Section~\ref{sec: 4 A category equivalence}.  
\end{definition}

\begin{lemma} \label{lem: basic patch}
Let $A$ be the patch algebra of a patch bundle $f:Y \rightarrow X$ over $R$. For each point $y \in Y$, there is a homomorphism of $A$ onto the ring $f(y)$, whose kernel is an idempotent generated ideal of $A$.   
\end{lemma}
\begin{proof}
Let $({\cal R},B)$ be the patch presheaf that is the image of the patch bundle $f$ under the functor ${\bf R}$. Specifically, $B = {\rm Clop}(Y)$. Let ${y} \in Y$. 
We define a ring homomorphism $\psi_{y}:A \rightarrow f(y)$  in the following way. First, let $\m$ be the maximal ideal corresponding to $y$ under Stone duality, i.e., $$\m = \{U \in {\rm Clop}(Y):y \not\in U\}=\{e\in B : y\not \in e\}.$$ By Lemma~\ref{lem: R presheaf}\ref{lem: R presheaf i}, every element $a\in A$ admits a full orthogonal decomposition $a=\sum_i t_ie_i \in A$, where the $e_i$ form a full orthogonal set in $B$ and the coefficients $t_i \in R_{e_i}$. Since $\m$ is a maximal ideal of $B$ and this decomposition is full orthogonal, it follows that exactly one of the $e_i$ is not in $\m$. 
With this said, we define now 
\[\psi_{y}(a) = t_i,\text{ where }i\text{ is such that }e_i \not \in \m.\] By Lemma~\ref{lem: bundle and stalks}, $X$ is the set of stalks of $({\cal R},B)$, so 
$ f (y)={\mathcal{R}}_\m=\bigcup_{e\not \in \m} R_e$ and hence  $\psi_{y}(a) \in {\mathcal{R}}_{\m}$. 

To see that $\psi_y(a)$ is well defined, suppose $$a = \sum_i t_i e_i = \sum_j s_j f_j$$ are full orthogonal decompositions of $a$ with $t_i \in R_{e_i}$ and $s_j \in R_{f_j}$.  Without loss of generality, we may assume $e_1,f_1 \not \in \m$ and $e_i,f_j \in \m$ for all $i,j \geq 2$.   Then $$ae_1f_1 =t_1e_1f_1=s_1e_1f_1.$$  
If $e_1f_1 =0$, then at least one of $e_1,f_1$ is in $\m$, contrary to assumption. Thus $e_1f_1 \ne 0$ with $(t_1 -s_1)e_1f_1=0$. From the faithfulness of $e_1f_1$ (see Proposition~\ref{prop: R presheaf}\ref{prop: R presheaf 2}) conclude that $t_1 = s_1$, proving that $\psi_y$ is well defined. 

We claim next that $\psi_{y}$ is a ring homomorphism. 
Let $a,b \in A$ and let $\sum a_j f_j$ and $\sum b_k f'_k$ be full orthogonal decompositions of $a$ and $b$, respectively. By multiplying $a$ by $\sum f'_k=1$ and $b$ by $\sum f_j=1$, we may refine the full orthogonal decompositions of $a$ and $b$ to full orthogonal decompositions sharing the same idempotents, say,  $a =\sum_i t_i e_i$ and $b = \sum_i s_ie_i$, with $t_i,s_i \in R_{e_i}$. As noted above, we may assume without loss of generality that $e_1\not \in \m$ and $e_i \in \m$ for all $i \geq 2$.  By definition, \begin{center} $\psi_y(a+b) = t_1+s_1 = \psi_y(a)+\psi_y(b)$ and $\psi_y(ab) = t_1s_1 = \psi_y(a)\psi_y(b)$.
\end{center}
Also, $\psi_y(1) =1$,  and so $\psi_y$ is a ring homomorphism.

To show that $\psi_y$ is onto, pick an arbitrary element $r\in f(y)={\cal R}_\m=\bigcup_{e\not \in \m} R_e$ (note that in our notation $e$ is a clopen set of $Y$). Then $r\in R_{e_0}$ for some clopen $e_0$ not in $\m$ and $a=r e_0=\psi_y^{-1}(r)$.

We claim finally that  the kernel of $\psi_y$ is $\m A$, i.e., the ideal of $A$ generated by $\m$.
It is easy to see that 
\begin{equation}\label{mA}
\m A = \bigcup_{e\in \m} eA.
\end{equation}
This follows from the fact that, if $e$ and $f$ are elements of $\m$ and $a, b\in A$, then $ea+fb=(e\vee f)(ea+fb)$, where $e\vee f \in \m$.

Let $a \in \ker \psi_y$, and write $a = \sum_i t_ie_i$, where the $e_i$ form a full orthogonal set of $B$ and $t_i \in R_{e_i}$ for every $i$. As noticed above, there is exactly one choice of $i$ for which $e_i \not \in \m$. For this choice of $i$, $t_i=0$, and so $a \in \m A$.  This proves $\ker \psi_y \subseteq \m A$. 

To prove the reverse inclusion, let $a \in \m A$. In light of \eqref{mA}, there is $e\in \m$ such that $a = e  \sum_i t_ie_i$, where the $e_i$ are pairwise orthogonal elements of $B$ and $t_i \in R_{e_i}$. Then, $a=\sum t_i f_i$, with $f_i=ee_i$ and $t_i\in R_{e_i}\subseteq R_{f_i}$. Since the $f_i$ are pairwise orthogonal and  all in $\m$, $\psi_y(a) =0$, proving that $\m A \subseteq \ker \psi_y$.
\end{proof} 

Following the terminology in \cite[Section 2]{BMO}, the {\it Pierce spectrum of a ring $R$} is the Stone space ${\rm Max}(\Id (R))$, while the {\it Pierce stalks of $R$} are the indecomposable rings $R/\m R$, where $\m\in {\rm Max}(\Id (R))$.  The next theorem, which is the main result of this section, relates patch bundles (and so patch spaces) of indecomposable rings with Pierce spectra and Pierce stalks of their patch algebras.

\begin{theorem} \label{thm: main indecomposable}
Let $R$ be an indecomposable ring. If $A$ is the patch algebra of a patch bundle  $f:Y \rightarrow X$ over $R$, then the Pierce spectrum of $A$ is homeomorphic to $Y$ and the rings in $X$ are, up to isomorphism, the Pierce stalks of $A$. 
\end{theorem}
\begin{proof}
Let $A$ be the patch algebra of the patch bundle  $f:Y \rightarrow X$. By Remark~\ref{rem: R indecomposable}, since $R$ is indecomposable, $\Id(A)=B={\rm Clop}(Y)$. Thus, we get from Theorem~\ref{thm: Stone's repr. thm} (Stone duality) that the Pierce spectrum of $A$, ${\rm Max}(\Id (A))={\rm Max}({\rm Clop}(Y))$, is homeomorphic to~$Y$. Moreover, by Lemma~\ref{lem: basic patch}, for each $y\in Y$ there is a homomorphism $\psi_y$ of $A$ onto the ring $f(y)$ whose kernel is generated by idempotents. In particular, inspection of the proof of the lemma shows that the kernel is generated by the maximal ideal $\m$ of $B$ associated to $y$ by the Stone duality. This proves that, up to isomorphism, the rings in $X$ (namely the $f(y)$) are the Pierce stalks of $A$.     
\end{proof}

In the case of Theorem~\ref{thm: main indecomposable}, there need not be a one-to-one correspondence between Pierce stalks of the patch algebra $A$ and rings in $X$, since the theorem does not rule out the possibility that  a  ring in $X$ can be isomorphic to more than one Pierce stalk. This is because $Y$, by enriching $X$ to a patch bundle, can ``overcount'' the rings in $X$. This phenomenon, however, does not occur when $Y = X$ is a patch space and $f$ is the identity patch bundle, as the next corollary shows.

 \begin{corollary} Let $R$ be an indecomposable ring,  let $X$ be a patch space of subrings of $R$, and let $A$ be the patch algebra of the identity patch bundle of $X$, as in  Remark~\ref{rem: patch spaces are bundles}. Then $X$ (with the patch topology) is  homeomorphic to the Pierce spectrum of $A$, 
 the rings in $X$ are, up to isomorphism, the Pierce stalks of $A$, and there is a one-to-one correspondence between Pierce stalks of $A$ and rings in $X$. 
\end{corollary}

\begin{proof}  
This is a direct application of Theorem~\ref{thm: main indecomposable} to the patch bundle $f: X_P \to X_Z \colon x\mapsto x$ defined in Remark~\ref{rem: patch spaces are bundles}. The fact that the correspondence between $X$ and the patch stalks is one-to-one follows from the fact that $X$ is homeomorphic to the Pierce spectrum of $X$.
\end{proof}

\section{Structure of patch algebras}\label{sec: 7 structure of patch algebras}

In order to better situate  patch algebras within commutative ring theory,  we consider the structure of these algebras for spaces consisting of special classes of rings. In the sequel to this paper, we will continue this investigation in more depth. 

 A (commutative) ring for which the annihilator of each element is generated by an idempotent is known in the literature as a {\it Rickart ring} \cite[Definition 7.45]{Lam99} or a {\it weak Baer ring} (see, e.g., \cite{BDMG11}). These are precisely the rings for which every principal ideal is projective \cite[Proposition~7.48]{Lam99}. With this characterization, Rickart rings are referred to in \cite{Bergman} as {\it p.p. rings}. Rickart rings can also be characterized by their Pierce stalks. It is proved in  \cite[Lemma 3.1]{Bergman} that a ring $R$ is Rickart if and only if the 
Pierce stalks of $R$ are domains and the support of every element of $R$ in the Pierce spectrum is clopen.  Here, the {\it support} of an element $a\in R$ is the set of elements $\m$ in the Pierce spectrum of $R$, such that $a\not \in \m R$.

 \begin{proposition}\label{prop: Rickart} A patch algebra of a patch bundle over a domain is a Rickart ring. 
 \end{proposition} 

 \begin{proof} Let $A$ be the patch algebra associated to the patch bundle $f:Y \rightarrow X$ of a domain $R$. Since the rings in $X$ (which are subrings of $R$) are domains, we get from Theorem~\ref{thm: main indecomposable} that so are the Pierce stalks of $A$. Using the characterization of Rickart rings in \cite[Lemma 3.1]{Bergman}, it remains to show that the support of every element of $A$ with respect to its Pierce spectrum is clopen.  
 
 Let $a \in A$, and as in Proposition~\ref{prop: unique full orth form}, write $a =\sum_i t_ie_i$, where each $t_i$ is a nonzero element $ R_{e_i}$, the $t_i$ are pairwise distinct, and the $e_i$ are non-zero idempotents of $A$ pairwise orthogonal. For a maximal ideal $\ff m$ of the Boolean algebra $\Id(A)$,  $a$ is in $\ff m A$ if and only if $e_i \in \m$ for all $i$. In fact, if all the $e_i$ are elements of $\m$, then $a=e a$, where $e=\bigvee_i e_i=\sum_i e_i$. On the other hand, if $a\in \m A$, then there is an idempotent element $e\in \m$ such that $a=ea$. Then, $t_i e_i= t_i e e_i$ from which we get $e_i=e e_i\in \m$. Therefore, the support of $a$ is the set of maximal ideals $\ff m$ of $\Id(A)$ that do not contain $e=\bigvee_i e_i=\sum_i e_i$ and hence is a clopen set. This proves that $A$  is a Rickart ring. 
 \end{proof}

A {\it semihereditary ring} is a ring in which every finitely generated ideal is projective. Clearly semihereditary rings are Rickart. In particular, they are exactly the Rickart rings whose Pierce stalks are Pr\"ufer \cite[Theorem 4.1]{Bergman}. The next result is then an easy consequence of the previous proposition. In the following, if $P$ is a ring-theoretic  property, we write {\it patch bundle of $P$ subrings of a ring $R$} to denote a patch bundle $f: Y\rightarrow X$ over $R$ in which $X$ is a space of subrings of $R$ satisfying property $P$.

\begin{corollary}\label{cor: semihereditary}
A patch algebra of a patch bundle of Pr\"ufer subrings of a domain is a semihereditary ring. 
\end{corollary}

\begin{proof}
 Let $A$ be the patch algebra associated to the patch bundle $f:Y \rightarrow X$ of a domain $R$. Since the rings in $X$ are Pr\"ufer domains, Theorem~\ref{thm: main indecomposable} ensures that so are the Pierce stalks of $A$. Since $A$ is a Rickart ring by Proposition~\ref{prop: Rickart}, we get from the above recalled characterization \cite[Theorem 4.1]{Bergman} due to Bergman, that $A$ is semihereditary.
\end{proof}

 Theorem~\ref{thm: main indecomposable} and Proposition~\ref{prop: Rickart} allow to describe the Pierce spectrum and stalks of the patch algebra associated to a patch bundle over a domain. We give the details on Proposition~\ref{prop: main domain}, after proving the following lemma. 

\begin{lemma}\label{lem: (minimal) primes}
   The following hold for any ring $R$.
\begin{enumerate}[label=\textup{(\roman{*})}, mode=unboxed]
    \item\label{lem: (minimal) primes i} If $\mathfrak p$ is a prime ideal of $R$, then $\mathfrak p \cap \Id(R)$ is an element of the Pierce spectrum of $R$.
    \item\label{lem: (minimal) primes ii}  If $\mathfrak p$ is a prime ideal of $R$ generated by idempotents, then it must be a minimal prime ideal of $R$.
\end{enumerate}
\end{lemma}

\begin{proof}
    \ref{lem: (minimal) primes i} Let $\mathfrak p$ be a prime ideal of $R$. It is routine to check that $I:= \mathfrak p \cap \Id(R)$ is an ideal of $\Id(R)$. Let $x$ and $y$ be two idempotents of $R$ such that $xy\in I\subseteq \mathfrak p$. Then, being $\mathfrak p$ prime, we have $x\in I$ or $y\in I$. This shows that $I\in {\rm Max}(\Id (R))$.

    \ref{lem: (minimal) primes ii} Let $\mathfrak p$ be a prime ideal of $R$ generated by idempotents, and let $\mathfrak q$ be a prime ideal of $R$ contained in $\mathfrak p$. We then get from item \ref{lem: (minimal) primes i} that $\m_1:=\mathfrak p \cap \Id(R)$ and $\m_2:=\mathfrak q \cap \Id(R)$ are maximal ideals of $\Id(R)$ such that $\m_2\subseteq \m_1$. This can happen only if $\m_1=\m_2$, i.e., only if $\mathfrak p$ and $\mathfrak q$ share the same idempotents. Since the generators of $\mathfrak p$ are then contained in $\mathfrak q$, we conclude that $\mathfrak q=\mathfrak p$, as desired.
\end{proof}

 \begin{proposition} \label{prop: main domain}
Let $R$ be a domain. If $A$ is the patch algebra of a patch bundle  $f:Y \rightarrow X$ of $R$, then the Pierce spectrum of $A$ is homeomorphic to the minimal spectrum of $A$ and the rings in $X$ are, up to isomorphism, the quotients of $A$ modulo its minimal prime ideals. 
\end{proposition}

\begin{proof}
Since $R$ is a domain, $R$ is indecomposable so, by Theorem \ref{thm: main indecomposable}, the rings in $X$ are isomorphic to the Pierce stalks of $A$. The Pierce stalks of $A$ are then domains, which implies that for every $\m$ in the Pierce spectrum of $A$, $\m A$ is a prime ideal of $A$ generated by idempotents. These prime ideals are then minimal prime ideals of $A$ by Lemma \ref{lem: (minimal) primes}\ref{lem: (minimal) primes ii}. Conversely, if $\mathfrak p$ is a minimal prime ideal of $A$, then $\mathfrak p \cap \Id(A)$ is an element of the Pierce spectrum of $A$ by Lemma \ref{lem: (minimal) primes}\ref{lem: (minimal) primes i}, and $\mathfrak p = (\mathfrak p \cap \Id(A)) A$. The mapping $\mathfrak p \rightarrow \mathfrak p \cap \Id(A)$ establishes then a bijection between the minimal spectrum ${\rm Min}(A)$ and the Pierce spectrum of $A$. This proves that the Pierce stalks of $A$ are the quotients of $A$ modulo its minimal prime ideals, as desired. Moreover, $A$ is a Rickart ring by Proposition~\ref{prop: Rickart} and it follows from \cite[Theorem~1]{Kist} that the above bijection is actually a homomorphism. This concludes the proof.
\end{proof}

Recall from \cite{Nic77} that a ring $R$ is {\it clean} if every element is the sum of an idempotent and a unit. Clean rings are precisely the rings whose Pierce stalks are local \cite[Theorem 3.4]{BurSte79}. The next result shows that a collection of local subrings of an indecomposable ring can be realized by factoring a single clean ring.

\begin{proposition}\label{proposition: A clean}
 If $A$ is a patch algebra of a patch bundle of local subrings of an indecomposable ring, then $A$ is a clean ring for which the rings in $X$ are, up to isomorphism, the localizations of $A$ at its maximal ideals. 
\end{proposition} 

\begin{proof}
 Since the rings in $X$ are local, Theorem~\ref{thm: main indecomposable} implies that the Pierce stalks of $A$ are local, i.e., that $A$ is a clean ring. Every element $\m$ in the Pierce spectrum of $A$ is contained in at least one maximal ideal $M$ of $A$. Thus, $A/\m A$ being local and $M/\m A$ being its maximal ideal, $\m$ is contained in exactly one maximal ideal $M$ of $A$. On the other hand, each maximal ideal of $A$ contains at most one element the Pierce spectrum of $A$ since the sum in $A$ of any two distinct elements of the Pierce spectrum contains the identity. This establishes a one-to-one correspondence between the Pierce spectrum ${\rm Max}(\Id(A))$ and ${\rm Max}(A)$. Let $M$ be a maximal ideal of $A$ and $\m$ be the unique ideal in the Pierce spectrum of $A$ contained in $M$. Since $A_M$ is indecomposable (being a local ring), $\m A_M=\{0\}$ and so the localization of $A/\m A$ at $M/\m A$ is isomorphic to $A_M$. Moreover, since  $A/\m A$ is local, this localization is isomorphic to $A/\m A$. Thus, $A_M$ is isomorphic to the Pierce stalk of $A$ at the ideal in the Pierce spectrum contained in $M$. 
\end{proof}

A ring $R$ is a {\it Gelfand (or PM) ring} if each prime ideal of $R$ is contained in a unique maximal ideal (see, e.g., \cite{Contessa}). Local rings and zero-dimensional rings (i.e., rings in which every prime ideal is maximal), are basic examples of Gelfand rings. 

\begin{corollary} \label{cor: local}
If $A$ is a patch algebra of a patch bundle of local subrings of a domain, then $A$ is a Gelfand ring whose maximal and minimal spectra are homeomorphic. 
\end{corollary}
\begin{proof}
Proposition \ref{prop: main domain} shows that the mapping $\mathfrak p \rightarrow \mathfrak p \cap \Id(A)$ establishes a homeomorphism between the minimal spectrum and the Pierce spectrum of $A$. This implies, by the discussion in the proof of the previous proposition, that each  minimal prime ideal of $A$ (and so each prime ideal of $A$) is contained in a unique maximal ideal of $A$. The patch algebra $A$ is then a Gelfand ring. This, together with the fact that $A$ is a Rickart ring (Proposition~\ref{prop: Rickart}), implies by \cite[Theorem 4.2 (d)]{Kist}, that the minimal and maximal spectrum of $A$ are homeomorphic. \end{proof}

The results in this section can be applied   to the spaces of rings discussed in Example~\ref{example: patch spaces}. As an example of how to do this, we give  a description of the patch algebra of the Zariski-Riemann space of a field. 

\begin{example} \label{example: ZS}
Let $F$ be a field, let $D$ be a subring of $F$, and let $X$ be the Zariski-Riemann space of $F/D$ that consists of the valuation rings of $F$ that contain $D$. As in Example~\ref{example: patch spaces}, $X$ is a patch space for $F$. 
Proposition~\ref{proposition: A clean} implies that the patch algebra $A$ of $X$ is a clean ring  with the property that the localizations of $A$ at its maximal  ideals are the valuation rings in $X$. (This in turn implies that $A$ is an arithmetical ring, i.e., a ring for which the ideals in each localization at a prime ideal form a chain.)  
Moreover, by Proposition~\ref{prop: main domain}, the rings in $X$ are the quotients of $A$ modulo its minimal prime ideals.  Finally, by this same proposition and  
 Corollary~\ref{cor: local}, $A$ is a Gelfand ring whose 
Pierce, minimal and maximal spectra of are all homeomorphic to $X$. 
\end{example}

\section*{Acknowledgments}
The authors are grateful to Giovanni Secreti, Valentin Havlovec and the anonymous referee for careful reading and helpful comments.

\end{document}